\documentclass[12pt,a4paper]{article}
\usepackage[T2A]{fontenc}
\usepackage[cp1251]{inputenc}
\usepackage{cmap}
\usepackage[english]{babel}
\usepackage{amsmath,amsfonts,amssymb, amsthm}
\usepackage{graphicx}
\usepackage{cite}
\usepackage{xcolor}
\usepackage{hyperref}
\usepackage[left=3cm,right=2cm, top=2cm,bottom=2cm,bindingoffset=0cm]{geometry}

\theoremstyle{plain}
\newtheorem{theorem}{Theorem}
\newtheorem{lemma}{Lemma}
\newtheorem{corollary}{Corollary}[theorem]

\newtheorem{question}{Question}
\theoremstyle{definition}
\newtheorem{example}{Example}

\usepackage[auth-sc]{authblk}

\title{On groups with modular Schmidt subgroups}
\author{Victor S. Monakhov}
\author{Irina L. Sokhor}
\affil{Department of Mathematics and Technologies of Programming,\\
Francisk Skorina Gomel State University, Belarus}

\affil{victor.monakhov@gmail.com, irina.sokhor@gmail.com}
\date{}
\begin{document}
\maketitle

\abstract{We prove that if every Schmidt subgroup of a group $G$ is subnormal or modular,
then $G/F(G)$ is cyclic.}



\section{Introduction}\label{sec1}

All groups in this paper are finite.

A group $G$ is called a Schmidt group
if $G$ is non-nilpotent and every proper subgroup of $G$ is nilpotent.
B.~Huppert gave the separate section of his monograph~\cite{hup}
to Schmidt groups.
V.\,S.~Monakhov~\cite{monsch} presented a review of the results about
properties of Schmidt groups, the existence of Schmidt subgroups
in finite groups and their some applications in the theory of classes
of finite groups.

Groups with subnormal Schmidt subgroups were investigated
in~\cite{knmon2004,ved2007}.
In particular, V.\,A.~Vedernikov proved the following statement.

\medskip

{\sl If all Schmidt subgroups of a group $G$ are
subnormal in $G$, then~$G/F(G)$ is cyclic,}~\cite{ved2007}.

\medskip

Here $F(G)$ is the Fitting subgroup of a group~$G$.

The results of~\cite{knmon2004,ved2007} developed in
various directions. Properties of groups with  partially subnormal Schmidt subgroups
were described in~\cite{hhskiba2021}. Groups with $\sigma$-subnormal Schmidt subgroups
were investigated in~\cite{as17,hh17,ky20,bky22}. 
Groups with modular Schmidt subgroups were studied in~\cite{hh17,blsel2019,hu2021}.
In~\cite{hhskiba2021}, it was proposed the problem to describe
the structure of groups with all Schmidt subgroups modular.
In particular, 

\begin{question}[{\cite[p.~175]{hhskiba2021}}]
Is the quotient group $G/F(G)$ is cyclic for a group $G$ with all Schmidt subgroups modular?
\end{question}




\medskip

The following theorem is our result.

\begin{theorem}\label{t1}
If every Schmidt subgroup of a group $G$ is
subnormal or modular, then $G/F(G)$ is cyclic.
\end{theorem}

\begin{corollary}
If every Schmidt subgroup of a group $G$ is modular, then $G/F(G)$ is cyclic.
\end{corollary}

\section{Preliminaries}

The concept of a modular subgroup came from the lattice theory.
A subgroup $M$  of a group $G$ is {\sl modular} in $G$ if
$M$ is a modular element of the lattice of subgroups of~$G$.
The detailed analysis of modular subgroups was conducted by R.~Schmidt~\cite{sch94}.

We say that a group $G$ is a $\mathrm{P}$-group
if $G$ is either elementary abelian of order $p^{n+1}$ for a prime $p$
or a semidirect product of an elementary abelian normal subgroup $A$
of order $p^n$ by a group of prime order $q$, $q\neq p$,
which induces a nontrivial power automorphism on~$A$ \cite[p.~49]{sch94}.

\begin{lemma}[{\cite[Theorem~5.1.14]{sch94}}]\label{lsch2}
A subgroup  $M$ of a group $G$ is modular in~$G$ if and only if
$$
G/M_G=S_1/M_G\times \ldots \times S_r/M_G\times T/M_G,
$$
where $r\in \mathbb N\cup \{0\}$ and for all $i,j\in \{1,\ldots ,r\}$,

$(1)$~$S_i/M_G$ is a nonabelian $\mathrm{P}$-group,

$(2)$~$(|S_i/M_G|,|S_j/M_G|)=(|S_i/M_G|,|T/M_G|)=1$ for $i\ne j$,

$(3)$~$M/M_G=Q_1/M_G\times\ldots \times Q_r/M_G\times (T\cap M)/M_G$
and $Q_i/M_G$ is a non-normal Sylow subgroup of $S_i/M_G$, and

$(4)$~$M\cap T$ is permutable in~$G$.
\end{lemma}

In the sequel, by $\mathfrak{A}_1$
we denote  a non-saturated formation
of all abelian groups with elementary abelian Sylow subgroups,
by $G^{\mathfrak{A}_1}$ we denote the $\mathfrak{A}_1$-residual
of a group~$G$, i.\,e. the smallest normal subgroup of $G$
with quotient in~$\mathfrak{A}_1$.

\begin{lemma}[{\cite[Lemma~4]{zim89}}]\label{lm1}
If $H$ is a modular subgroup of a group~$G$, then $H^\prime $ and $H^{\mathfrak{A}_1}$ are subnormal in~$G$.
\end{lemma}


\begin{lemma}
[{\cite{hup,monsch,knmon2004,ved2007}}]\label{lsch}
Let $S$ be a Schmidt group. The following statements hold.

$(1)$~$S=P\rtimes Q$, where $P$ is a normal Sylow $p$-subgroup
and $Q=\langle y\rangle$ is a non-normal Sylow $q$-subgroup
for distinct primes $p$ and $q$.

$(2)$~$S^\prime =P$, $\Phi (S)=\Phi (P)\times \langle y^q\rangle$,
$y^q\in Z(S)$.


$(3)$~$S$ is supersoluble if and only if $|P|=p$ and $q$ divides~$p-1$.
\end{lemma}

Following~\cite{monsch},
we denote by $S_{\langle p,q \rangle}$ the class of all Schmidt groups
with a normal Sylow $p$-subgroup and a non-normal cyclic Sylow $q$-subgroup.
Therefore every $S_{\langle p,q \rangle}$-group
can be represented as $P\rtimes Q$,
where $P$ is a normal Sylow $p$-subgroup and
$Q=\langle y\rangle$ is a cyclic non-normal Sylow $q$-subgroup.

\begin{lemma}\label{tsch}
Let $M=P\rtimes Q$ be a modular $S_{\langle p,q\rangle}$-subgroup
of a group~$G$. The following statements hold.

$(1)$~$P\times \langle y^q\rangle \le F(G)$, where
$\langle y\rangle =Q$.

$(2)$~If $M$ is not subnormal in~$G$,
then $G/M_G$ is a non-abelian $\mathrm{P}$-group
of order $p_1^nq$ for primes $p_1>q$ and $|G|=|M|p_1^n$.
\end{lemma}

\begin{proof}
$(1)$~In view of Lemma~\ref{lsch}\,(2), $M^\prime =P$,
therefore
$M^{\mathfrak{A}_1}=P\times \langle y^q\rangle$.
By Lemma~\ref{lm1},
$P\times \langle y^q\rangle$ is subnormal in~$G$.
Since $P\times \langle y^q\rangle$ is nilpotent,
we get~$P\times \langle y^q\rangle \le F(G)$.

$(2)$~Since $M_G<M$ and $M/M_G$ is nilpotent by~\cite[Theorem~5.1.14]{sch94},
we deduce that $P\le M_G$ and $M/M_G=(P\rtimes Q)/M_G\cong Q/(Q\cap M_G)$
is a cyclic $q$-group. Therefore $M/M_G$ does not
decompose into the direct product of two proper subgroups.
Using Lemma~\ref{lsch2}\,(3), we get $r=0$ or~$r=1$.

If $r=0$, then $G/M_G=T/M_G$, and according to Lemma~\ref{lsch2}\,$(4)$,
$M$ is permutable. But a permutable subgroup is subnormal by~\cite[Theorem~5.1.1]{sch94},
a contradiction. Therefore $r=1$.
From Lemma~\ref{lsch2}\,$(3)$ it follows that $T/M_G=1$ and
$M/M_G=Q_1/M_G$ is a non-normal Sylow $q_1$-subgroup of $S_1/M_G$,
in particular, $q_1=q$. In view of Lemma~\ref{lsch2}\,$(1)$,
$G/M_G$ is a non-abelian $\mathrm{P}$-group.
\end{proof}

\begin{example}
In the dihedral group $D_{30}$ of order $30$,
a non-subnormal Schmidt subgroup~$S_3$ is modular and $|D_{30}|=|S_3|\cdot 5$.
Therefore in Lemma~\ref{tsch}\,(2), $p_1\neq p$ in general.
\end{example}

\section{Proof of Theorem~{\upshape\ref{t1}}}
If every Schmidt subgroup of a group $G$ is subnormal in $G$,
then $G/F(G)$ is cyclic~\cite[Corollary~1]{ved2007}.
Assume that $G$ contains a non-subnormal
$S_{\langle p,q\rangle}$-subgroup
$M=P\rtimes Q$. By the choice of~$G$,
$M$ is modular in~$G$. In view of Lemma~\ref{tsch}\,(2),
$G/M_G$ is a non-abelian $\mathrm{P}$-group of order
$p_1^nq$, $p_1>q$, and $|G|=|M|p_1^n$.
Since $|\pi (M)|=2$, we get $|\pi (G)|\le 3$.

Let $|\pi (G)|=2$. Then $p_1=p$ and $|G|=|P|p^n|Q|$.
By the definition of a $\mathrm{P}$-group,
$q$ divides $(p-1)$. Hence $M$ is supersoluble by Lemma~\ref{lsch}\,(3),
and $|P|=p$. Since a $\mathrm{P}$-group is supersoluble and $P\le F(G)$,
then~$G$ is $p$-closed. As $Q$ is a Sylow subgroup of~$G$,
we get $G/F(G)$ is cyclic.

Let $|\pi (G)|=3$. Then $G=MP_1=(P\rtimes Q)P_1$, $p\ne p_1$,
and $P_1$ is a Sylow $p_1$-subgroup of~$G$.
By the definition of a $\mathrm{P}$-group,
$P_1$ is an elementary abelian subgroup and
$q$ divides $(p_1-1)$. We can assume without loss of
generality that  $P_1Q$ is a subgroup of~$G$.
Since $q$ divides $(p_1-1)$, we get $P_1Q$ is supersoluble.
As $P_1$ is an elementary abelian subgroup,
by Maschke's theorem,
there are elements $a_1,\ldots ,a_n$ such that
$$
P_1=\langle a_1\rangle \times \ldots \times \langle a_n\rangle
$$
and $\langle a_i\rangle Q$ is a subgroup for every~$i$.
In view of Lemma~\ref{tsch}\,(2),
$$
G/M_G=(P_1M_G/M_G)\rtimes (QM_G/M_G)
$$
is a $\mathrm{P}$-group.
Hence~$QM_G/M_G$ induces a non-trivial power automorphism on
$P_1M_G/M_G$, in particular,
$$
(\langle a_i\rangle M_G/M_G)\rtimes (QM_G/M_G)
$$
is not nilpotent for every~$i$.
Therefore $\langle a_i\rangle Q$ is not nilpotent
and there is a Schmidt subgroup~$S_i=\langle a_i\rangle \rtimes Q_i$.
By the choice of~$G$, $S_i$ is subnormal or modular in~$G$.
In view of Lemma~\ref{lsch}\,(2), $\langle a_i\rangle=S_i^\prime$.
If $S_i$ is subnormal in~$G$,
then $\langle a_i\rangle$  is also subnormal in~$G$.
Let $S_i$ be modular in~$G$. By Lemma~\ref{lm1},
$\langle a_i\rangle$ is subnormal in~$G$.
Thus for every $i$, $\langle a_i\rangle$ is subnormal in~$G$.
Hence $P_1=\langle a_1\rangle \times \ldots
\times \langle a_n\rangle$ is normal in~$G$ and $PP_1\le F(G)$.
But now $G/F(G)\cong QF(G)/F(G)$ is cyclic.


\begin{thebibliography}{9}
\bibitem{hup}
B.~Huppert, Endliche Gruppen. I,
Springer-Verlag, Berlin--Heidelberg--New York, 1967.

\bibitem{monsch}
V.~S.~Monakhov,
The Schmidt subgroups, its existence, and some of their
applications. In: Proceedings of the Ukrainian Mathematical Congress-2001,
Inst. Mat. NAN Ukrainy, Kyiv (2002), 81--90.

\bibitem{knmon2004}
V.~N.~Knyagina, V.~S.~Monakhov,
Finite groups with subnormal Schmidt subgroups,
Sib. Math. J., 45(6) (2004) 1075--1079.

\bibitem{ved2007}
V.~A.~Vedernikov,
Finite groups with subnormal Schmidt subgroups,
Algebra Logic, 46(6) (2007) 363--372.

\bibitem{hhskiba2021}
J.~Huang, B.~Hu, A.\,N.~Skiba, A.N.,
Finite Groups with Weakly Subnormal and Partially Subnormal Subgroups,
Sib. Math. J., 62\,(1) (2021) 169--177.


\bibitem{as17}
K.~A.~Al-Sharo, A.~N.~Skiba,
On finite groups with $\sigma$-subnormal Schmidt subgroups,
Comm. Algebra, 45(10)  (2017) 4158--4165.



\bibitem{ky20}
X.~Yi, S.~F.~Kamornikov,
Finite groups with $\sigma$-subnormal Schmidt subgroups,
J. Algebra, 560 (2020) 181--191.


\bibitem{bky22}
A.~Ballester-Bolinches, S.~F.~Kamornikov, X.~Yi,
Finite groups with $\sigma$-subnormal Schmidt subgroups,
Bull. Malays. Math. Sci. Soc., 45 (2022) 2431--2440.

\bibitem{hh17}
B.~Hu, J.~Huang,
On finite groups with generalized $\sigma$-subnormal Schmidt subgroups,
Comm. Algebra, 46(7) (2018) 3127--3134.

\bibitem{blsel2019}
I.~V.~Bliznets, V.~M.~Selkin,
On finite groups with modular Schmidt subgroup,
Prob. Phis. Math. Tech., 4(41) (2019) 36--38.

\bibitem{hu2021}
B.~Hu, J.~Huang, D.~Song, I.~N.~Safonova,
Finite groups with $K$-$\mathfrak{F}$-subnormal Schmidt subgroups,
Comm. Algebra, 49(10) (2021) 4513--4518.

\bibitem{sch94}
R.~Schmidt,
Subgroup Lattices of Groups,
De Gruyter, Berlin--New York, 1994.

\bibitem{zim89}
I.~Zimmermann,
Submodular subgroups in finite groups,
Math. Z., 202 (1989) 545--557. 
\end{thebibliography}
\end{document}